\documentclass{article}
\usepackage{amssymb}
\usepackage{amsthm}
\usepackage{amsmath}
\usepackage{url}
\newtheorem{theorem}{Theorem}[section]
\newtheorem{corollary}{Corollary}[section]
\newtheorem{lemma}{Lemma}[section]
\newcommand{\R}{\mathbb{R}}

\title{Extending tensors on polar manifolds}
\author{Ricardo A. E. Mendes}

\begin{document}
\maketitle

\section{Introduction}

Let $(M,g)$ be a Riemannian manifold and $G$ a Lie group acting on $M$ properly by isometries. 
Recall that, by definition (see \cite{PalaisTerng87}, \cite{GroveZiller12}), this action is called \emph{polar} if there exists an immersed sub-manifold
 $\Sigma\to M$ meeting all $G$-orbits orthogonally. Such a
 submanifold $\Sigma$ is called a \emph{section}, and comes with a
 natural action by a discrete group of isometries $W=W(\Sigma)$, called its
 \emph{generalized Weyl group}. Sections are always totally geodesic, and the immersion $\Sigma \to M$ induces an isometry $\Sigma /W \to M/G$, so in particular $M/G$ is a Riemannian orbifold.

Denote by $C^\infty(T^{k,l}M)^G$, respectively $C^\infty(T^{k,l}\Sigma)^{W(\Sigma)}$, the sets of smooth $(k,l)$-tensors on $M$, respectively $\Sigma$, which are invariant under $G$, respectively $W$. Our main result states that the natural restriction map $C^\infty(T^{k,l}M)^G\to C^\infty(T^{k,l}\Sigma)^{W(\Sigma)}$ is surjective:

\begin{theorem}
\label{mainthm}
 Let $M$ be a polar $G$-manifold with immersed section $i:\Sigma\to M$, and $W(\Sigma)$ the generalized Weyl group associated to $\Sigma$.
Define the pull-back (restriction) map 
$$ \psi=i^* :C^\infty(T^{k,l}M)^G\to C^\infty(T^{k,l}\Sigma)^{W(\Sigma)}$$
by 
$$[\psi(\beta)](x)(v_1,\ldots v_l)=P^{\otimes k}[\beta(i(x)((di)_x v_1, \ldots (di)_x v_l)] $$
where $P:T_{i(x)}M\to T_x \Sigma$ is orthogonal projection.
Then $\psi$ is surjective.
\end{theorem}

In the case of functions, that is, $(k,l)=(0,0)$, the map $\psi$ above is an isomorphism. This is known as the Chevalley Restriction Theorem --- see \cite{PalaisTerng87}.

Note that Theorem \ref{mainthm} applies to $(0,l)$-tensors with symmetry properties, such as symmetric $l$-tensors, exterior $l$-forms, etc. This can be phrased naturally in terms of Weyl's construction (see \cite{FultonHarris} Lecture 6). Recall that Weyl's construction associates to each partition $\lambda=(\lambda_1, \ldots \lambda_k)$ of $l\in\mathbb{N}$ a functor $\mathbb{S}_\lambda$ of vector spaces called its Schur functor. One recovers $\Lambda^l$ and Sym$^l$ as the Schur functors associated to $\lambda=(l)$ and $\lambda=(1,1, \ldots 1)$, respectively. 

\begin{corollary}
\label{symmetrizer}
Let $M$ be a Riemannian manifold with an isometric polar action by $G$.
Let $\lambda=(\lambda_1, \ldots , \lambda_k)$ be a partition of $l\in\mathbb{N}$, and consider the associated Schur functor $\mathbb{S}_\lambda$. Then the (surjective) restriction map $\psi: C^\infty(T^{0,l}M)^G\to C^\infty(T^{0,l}\Sigma)^W$ induces a surjective map 
$$ \psi_\lambda:C^\infty(\mathbb{S}_\lambda (T^*M))^G \to C^\infty(\mathbb{S}_\lambda (T^*\Sigma))^W $$
\end{corollary}

For context, consider two special cases of Corollary \ref{symmetrizer}: exterior $l$-forms and symmetric $2$-tensors. In the case of exterior forms the conclusion of Corollary \ref{symmetrizer} is implied by P. Michor's Basic Forms Theorem --- see \cite{Michor96} and \cite{Michor97}. In fact, Michor's Theorem gives more precise information: it states that for a polar $G$-manifold $M$ with section $\Sigma$, every smooth $W(\Sigma)$-invariant $l$-form on $\Sigma$ can be extended \emph{uniquely} to a smooth $G$-invariant $l$-form on $M$ which is \emph{basic}, that is, vanishes when contracted to vectors tangent to the $G$-orbits.

The case of symmetric $2$-tensors follows from \cite{Mendes11}, which is again a sharper statement in the sense that a set of basic tensors is identified. This is used in the following extension result for Riemannian metrics: 
\begin{theorem}
\label{metric}
Let $G$ act polarly on the Riemannian manifold $M$ with section $\Sigma$ and generalized Weyl group $W$. Consider the restriction map (which is surjective by Corollary \ref{symmetrizer}):
$$ \psi=|_\Sigma: C^\infty(Sym^2M)^G \to C^\infty(Sym^2\Sigma)^W$$
For any Riemmanian metric $\sigma\in C^\infty(Sym^2\Sigma)^W $, there is a Riemannian metric $\tilde{\sigma}\in C^\infty(Sym^2M)^G$ such that $\psi(\tilde{\sigma})=\sigma$, and with respect to which the $G$-action is polar with the same section $\Sigma$.
\end{theorem}

For both Theorem \ref{metric} and Michor's Basic Forms Theorem, the proof relies on polarization results in the Invariant Theory of finite reflection groups --- see section \ref{polarizations}. On the other hand, the main ingredient in the proof of Theorem \ref{mainthm} is a multi-variable version of the Chevalley Restriction Theorem due to Tevelev --- see section \ref{MVCRT}.

An application of Theorem \ref{metric} is to give a partial answer to a natural question by K. Grove: Given a proper isometric action of $G$ on a Riemannian manifold $(M,g)$, describe the set of all metrics on $M/G$ which are induced by smooth $G$-invariant metrics $g_0$ on $M$. Theorem \ref{metric} answers this question under the additional hypothesis that $M$ is a polar $G$-manifold. Namely, that set of metrics on $M/G=\Sigma/W$ coincides with the set of smooth orbifold metrics.

Another application of Theorem \ref{metric} is an important step in the main reconstruction result in \cite{GroveZiller12}. This was in fact our main motivation for Theorem \ref{metric}.

The present paper is organized as follows.

In section \ref{MVCRT} we state Tevelev's multi-variable version of the Chevalley Restriction Theorem for isotropy representations of symmetric spaces (Theorem \ref{Tevelev}), and generalize it to the class of polar representations (Corollary \ref{polarMVCRT}).

Section \ref{sectionpolar} is concerned with the proofs of Theorem \ref{mainthm} and Corollary \ref{symmetrizer}.

In section \ref{polarizations} we show how the algebraic results behind Michor's Basic Forms Theorem \cite{Michor96, Michor97} and Theorem \ref{metric} (namely Solomon's Theorem \cite{Solomon63} and Theorem \ref{hessian}) are in fact results about polarizations in the Invariant Theory of finite reflection groups. We then show in detail how Theorem \ref{metric} follows from Theorem \ref{hessian}.

The Appendix provides a proof of Theorem \ref{hessian}. It is computer-assisted, and mostly reproduced from the author's PhD dissertation \cite{Mendes11}.

Acknowledgements: Part of this work was completed during my PhD, and I would like to thank my advisor W. Ziller for the long-term support. I would also like to thank A. Lytchak and J. Tevelev for useful communication.

\section{Multi-variable Chevalley Restriction Theorem}
\label{MVCRT}

Let $(G,K)$ be a symmetric pair, and consider the isotropy representation of $K$ on $V=T_K G/K$, also called an s-representation. This is polar, and any maximal abelian sub-algebra $\Sigma\subset V$ is a section.  Its generalized Weyl group $W$ is also called the ``baby Weyl group''. The classic Chevalley Restriction Theorem says that 
$$ |_\Sigma : \R[V]^K\to \R[\Sigma]^W$$
is an isomorphism (see \cite{Warner}, page 143).

Now consider the diagonal action of $K$ on $V^m$ (respectively  $W$ on $\Sigma^m$), and the corresponding algebras of invariant ($m$-variable) polynomials $\R[V^m]^K$ (respectively $\R[\Sigma^m]^W$). In contrast with the single-variable case, the restriction map $|_\Sigma $
is not injective. On the other hand, surjectivity is due to Tevelev:
\begin{theorem}[\cite{Tevelev00}]
\label{Tevelev} In the notation above, the restriction map
$|_\Sigma:\R[V^m]^K\to \R[\Sigma^m]^W $
is surjective. 
\end{theorem}

Remarks: The proof of Theorem \ref{Tevelev} relies on the Kumar-Mathieu Theorem, previously known as the PRV conjecture, see \cite{Kumar89} and \cite{Mathieu89}. Joseph \cite{Joseph97} proved the Theorem above in the special case of the adjoint action, using similar techniques. In \cite{Tevelev00} the Theorem above is stated only for $m=2$ factors. But on page 324 it is remarked that ``Actually, this (and Josephs's) Theorem also holds for any number of summands [...] ''.

We observe that Theorem \ref{Tevelev} generalizes to the class of \emph{polar} representations. (See \cite{Dadok85} for a treatment of polar representations)

\begin{corollary}
\label{polarMVCRT}
Let $K\subset O(V)$ be a polar representation, with section $\Sigma$ and generalized Weyl group $W\subset O(\Sigma)$. Then the $m$-variable restriction is surjective: 
$$|_\Sigma :\R[V^m]^K\to \R[\Sigma^m]^W $$
\end{corollary}
\begin{proof}
Let $K_0$ be the connected component of $K$ which contains the identity. It is polar with the same section $\Sigma$. Let $W_0$ be its generalized Weyl group, so that $W_0\subset W$. 
From the classification of irreducible polar representations in \cite{Dadok85}, it follows that the maximal subgroup $\tilde{K}\subset O(V)$, containing $K_0$, that is orbit-equivalent to $K_0$, defines an s-representation. (This fact has been given a classification-free proof in \cite{EschenburgHeintze99}.) Note that $K_0$ and $\tilde{K}$ have the same sections and generalized Weyl groups.

Theorem \ref{Tevelev} states that $$|_\Sigma:\R[V^m]^{\tilde{K}}\to\R[\Sigma^m]^{W_0}$$ is surjective. But since $\tilde{K}\supset K_0$, we have $\R[V^m]^{\tilde{K}}\subset \R[V^m]^{K_0}$, and so 
$$|_\Sigma:\R[V^m]^{K_0}\to\R[\Sigma^m]^{W_0}$$
is again surjective.

Finally, to show $|_\Sigma:\R[V^m]^{K}\to\R[\Sigma^m]^{W}$ is surjective, let $\beta\in\R[\Sigma^m]^{W}$. Then there is $\tilde{\beta_0}\in \R[V^m]^{K_0}$ which restricts to $\beta$. Define
$$ \tilde{\beta}=\frac{1}{|K/K_0|}\sum_{h\in K/K_0} h \tilde{\beta_0}$$
Since $\tilde{\beta}$ equals the average of $\tilde{\beta_0}$ over $K$, it is $K$-invariant. To show that $\tilde{\beta}|_\Sigma=\beta$, we note that each coset $hK_o\in K/K_0$ can be represented by some $h\in N(\Sigma)$. Indeed, for an arbitrary $h\in K$, $h\Sigma$ is a section for $K$, hence also for $K_0$. Since $K_0$ acts transitively on the sections, there is $h_0\in K_0$ such that $hh_0^{-1}\in N(\Sigma)$. Therefore
$$ \tilde{\beta}|_\Sigma= \frac{1}{|K/K_0|}\sum_{h\in K/K_0} (h \tilde{\beta_0})|_\Sigma=\frac{1}{|K/K_0|}\sum \beta= \beta $$
because $\beta$ is $W$-invariant.
\end{proof}

Note that the algebra of multi-variable polynomials $\R[V^m]$ is graded by $m$-tuples of natural numbers  $(d_1,\ldots,d_m)$, and similarly for  $\R[\Sigma^m]$. Consider the subspace generated by the polynomials of degree $(*,1,\ldots, 1)$. These can be identified with those tensor fields of type $(0,m-1)$ which have polynomial coefficients, that is, members of $\R[V, (V^*)^{m-1}]$, respectively $\R[\Sigma, (\Sigma^*)^{m-1}]$.

Since this grading is preserved by the restriction map $|_\Sigma$, Corollary \ref{polarMVCRT} implies:
\begin{corollary}
\label{maincor}
Let $K\subset O(V)$ be a polar representation, with section $\Sigma$ and generalized Weyl group $W\subset O(\Sigma)$. Then the restriction map for polynomial-coefficient invariant $(0,l-1)$-tensors
$$ |_\Sigma :\R[V, (V^*)^{l-1}]^K\to \R[\Sigma, (\Sigma^*)^{l-1}]^W$$
is surjective.
\end{corollary}

\section{Extending tensors}

\label{sectionpolar}

The goal of this section is to provide proofs of Theorem \ref{mainthm} and Corollary \ref{symmetrizer}.  We start with two Lemmas that will be used in proving Theorem \ref{mainthm}.

\begin{lemma}
\label{polarrep}
Let $V$ be a polar $K$-representation with section $\Sigma$ and generalized Weyl group $W$. Then restriction to $\Sigma$ is a surjective map
$$|_\Sigma :C^\infty(T^{0,l}V)^K \to C^\infty(T^{0,l}\Sigma)^W $$
\end{lemma}
\begin{proof}
The space of polynomial-coefficient $(0,l)$-tensors $\R[V,(V^*)^l]^K\subset C^\infty(T^{0,l}V)^K$ is generated, as an $\R[V]^K$-module, by finitely many (homogeneous) $\sigma_1, \ldots \sigma_r$. (See \cite{Springer} Proposition 2.4.14)

Since $\R[V]^K=\R[\Sigma]^W$, Corollary \ref{maincor} implies that the restrictions $\sigma_1|_\Sigma, \ldots \sigma_r|_\Sigma$ generate  $\R[\Sigma,(\Sigma^*)^l]^W$ as an $\R[\Sigma]^W$-module.

Then, by an argument involving the Malgrange Division Theorem  and the fact that $\R[\Sigma,(\Sigma^*)^l]^W$ is dense in $C^\infty(T^{0,l}\Sigma)^W$ (see \cite{Field77} Lemma 3.1), we conclude that $\sigma_1|_\Sigma, \ldots \sigma_r|_\Sigma$ generate $C^\infty(\Sigma, (\Sigma^*)^l)^W=C^\infty(T^{0,l}\Sigma)^W$ as a $C^\infty(\Sigma)^W$-module. This implies that $|_\Sigma :C^\infty(T^{0,l}V)^K \to C^\infty(T^{0,l}\Sigma)^W $ is surjective.
\end{proof}

The next Lemma describes the smooth $G$-invariant tensors on a tube $\mathcal{U}=G\times_K V$ in terms of smooth $K$-invariant tensors on the slice $V$. 
\begin{lemma}
\label{tube}
Let $K\subset G$ be Lie groups with $K$ compact, and $V$ be a $K$-representation. Define $\mathcal{U}=G\times_K V$ to be the quotient of $G\times V$ by the free action of $K$ given by $k\cdot (g,v)= (g k^{-1}, kv)$, and identify $V$ with the subset of $\mathcal{U}$ which is the image of $\{1\}\times V\subset G\times V$ under the natural quotient projection $G\times V \to \mathcal{U}$.

Then there is a $K$-representation $H$ and an isomorphism
$$ C^\infty(T^{0,l} V)^K\times C^\infty (V,H)^K \to C^\infty(T^{0,l}\mathcal{U})^G$$
Under this identification the restriction map $$|_V :C^\infty(T^{0,l}\mathcal{U})^G \to C^\infty(T^{0,l} V)^K $$
corresponds to projection onto the first factor. In particular $|_V$ is onto.
\end{lemma}
\begin{proof}
To describe $H$, let $p\in \mathcal{U}$ be the image of $(1,0)\in G\times V$ in $\mathcal{U}$. Then $(V^*)^{\otimes l}$ is a $K$-invariant subspace of $(T^*_p\mathcal{U})^{\otimes l}$, and we define $H$ to be its $K$-invariant complement, so that
$$ (T^*_p\mathcal{U})^{\otimes l}= (V^*)^{\otimes l} \oplus H$$  
as $K$-representations.

We define $\Psi: C^\infty(T^{0,l} V)^K\times C^\infty (V,H)^K \to C^\infty(T^{0,l}\mathcal{U})^G $ in the following way: Given $(\beta_1 ,\beta_2)\in C^\infty(T^{0,l} V)^K\times C^\infty (V,H)^K$, let $\tilde{\beta}:G\times V\to T^{0,l}\mathcal{U}$ be given by
$$ \tilde{\beta} (g,v)= g\cdot (\beta_1 (v) +\beta_2 (v)) $$
Since $\tilde{\beta}$ is $K$-invariant, it descends to $\beta = \Psi(\beta_1, \beta_2) :\mathcal{U}\to T^{0,l}\mathcal{U} $.

The map $\beta$ is smooth because $\tilde{\beta}$ is smooth and the action of $K$ on $G\times V$ is free. Moreover $\beta$ is clearly a $G$-invariant cross-section of the bundle $T^{0,l}\mathcal{U}\to \mathcal{U}$, and $\beta |_V = \beta_1$.
\end{proof}

Now the proof of Theorem \ref{mainthm} essentially follows from Lemmas \ref{polarrep}, \ref{tube}, together with the Slice Theorem (see \cite{Bredon}) and partitions of unity:
\begin{proof}[Proof of Theorem \ref{mainthm}]
First note that it is enough to consider $(0,l)$ tensors. Indeed, $\psi$ for $(k,l)$ tensors equals the composition of $\psi$ for $(0,k+l)$-tensors with raising and lowering indices (using the Riemannian metric on $M$) to transform between $(k,l)$-tensors and $(0,k+l)$-tensors.

It is enough to prove surjectivity of $\psi$ locally around each orbit in $M$, because of the existence of $G$-invariant partitions of unity subject to any cover by $G$-invariant open sets in $M$.

So let $p\in M$ be an arbitrary point, with orbit $Gp$, isotropy $K=G_p$, and slice $V=(T_pGp)^\perp$. The Slice Theorem (see \cite{Bredon}) then says that for an open $G$-invariant tubular neighborhood $\mathcal{U}$ of the orbit $Gp$ there is a $G$-equivariant diffeomorphism 
$$ E: G\times_K V\to \mathcal{U}$$
From now on we we will identify $\mathcal{U}$ with $G\times_K V$ through $E$.

The slice representation of $K$ on $V$ is polar (see \cite{PalaisTerng87}). If $\Sigma\subset V$ is a section with generalized Weyl group $W(\Sigma)$, the quotients $\mathcal{U} /G$, $V/K$ and $\Sigma/W$ are isometric. 

Since the inclusion $\Sigma\to\mathcal{U}$ factors as $\Sigma\to V\to \mathcal{U}$, the restriction map $\psi$ factors as $\psi=|_\Sigma^V\circ |_V^\mathcal{U}$, where
$$ |_\Sigma^V:C^\infty(T^{0,l}V)^K \to C^\infty(T^{0,l}\Sigma)^W \qquad |_V^\mathcal{U}: C^\infty(T^{0,l}\mathcal{U})^G\to C^\infty(T^{0,l}V)^K $$
Both these maps are surjective, by Lemmas \ref{polarrep} and \ref{tube}. Therefore $\psi$ is surjective.
\end{proof}

Now we turn to Corollary \ref{symmetrizer}, about $(0,l)$-tensors with symmetry properties, such as exterior forms and symmetric tensors.
\begin{proof}[Proof of Corollary \ref{symmetrizer}]
The Schur functor $\mathbb{S}_\lambda$ is defined in terms of a certain element $c_\lambda \in \mathbb{Z}S_l$ in the group ring $\mathbb{Z}S_l$, called the Young symmetrizer associated to $\lambda$ --- see \cite{FultonHarris} Lecture 6. Indeed, given a vector space $V$, the group $S_l$ acts on $V^{\otimes l}$, and so $c_\lambda$ determines a linear map $V^{\otimes l} \to V^{\otimes l}$. The image of this map is defined to be $\mathbb{S}_\lambda (V)$.

Thus $C^\infty (\mathbb{S}_\lambda (T^*M))$ is simply the image of the natural map $$c_\lambda:C^\infty(T^{0,l}M)\to C^\infty(T^{0,l}M)$$ and similarly for $C^\infty (\mathbb{S}_\lambda (T^*M))^G$ (because the actions of $G$ and $S_l$ commute), and $C^\infty (\mathbb{S}_\lambda (T^*\Sigma))^W$.

Since the restriction map $\psi$ is $S_l$-equivariant and surjective, it takes the image of $$c_\lambda:C^\infty(T^{0,l}M)^G\to C^\infty(T^{0,l}M)^G$$ onto the image of $$c_\lambda:C^\infty(T^{0,l}\Sigma)^W\to C^\infty(T^{0,l}\Sigma)^W$$
completing the proof.
  \end{proof}

\section{Polarizations and finite reflection groups}
\label{polarizations}
An alternative way of proving special cases of Theorem \ref{Tevelev} is given by the polarization technique. This has the advantage of providing explicit lifts, which we exploit to give a proof of Theorem \ref{metric}.

We start by recalling the definition of polarizations (see \cite{Schwarz07} for a reference). Let $U$ be an Euclidean vector space, and $H\to O(U)$ be a representation of the group $H$. Consider the diagonal action of $H$ on $m$ copies of $U$, and the corresponding algebra of invariant ($m$-variable) polynomials $\R[U^m]^H$. Identify $\R[U]^H$ with the elements of $\R[U^m]^H$ which depend only on the first variable.

The method of polarizations consists of generating multi-variable invariants from single-variable invariants. Indeed, assuming $f\in\R[U]^H$ is homogeneous of degree $d$, let $t_1, \ldots t_m$ be formal variables, and formally expand 
$$f(t_1v_1 +\ldots + t_mv_m)=\sum_{r_1+\ldots +r_m=d}t_1^{r_1}\cdots t_m^{r_m} f_{r_1, \ldots, r_m}(v_1, \ldots, v_m)$$
Then each $f_{r_1, \ldots, r_m}$ belongs to $\R[U^m]^H$, and is called a polarization of $f$. 

An alternative but equivalent definition of polarizations is given in terms of \emph{polarization operators} --- see \cite{Wallach93}. These are differential operators $D_{ij}$ (for $1\leq i,j\leq m$) on $\R[U^m]^H$ defined by
$$ (D_{ij} f ) (u_1, \ldots u_m)= \left. \frac{d}{dt}\right|_{t=0} f(u_1, \ldots, u_j+tu_i, \ldots u_m) $$
Then one defines the subalgebra $\mathcal{P}^m\subset \R[U^m]^H$ of polarizations to be the smallest subalgebra of $\R[U^m]^H$ containing $\R[U]^H$ and stable under the operators $D_{ij}$.

For example, if $f\in\R[U]^H$, then the tensors $df=D_{2,1}f \in \R[U^2]^H$ and Hess$f=D_{2,1}(D_{3,1} f)\in\R[U^3]^H$ are polarizations. Similarly, if $f_1, \ldots f_p \in \R[U]^H$, then $df_1\otimes df_2\otimes \cdots \otimes df_p=(D_{2,1}f_1)\cdots(D_{p+1,1}f_p) $ is a polarization, and so is $df_1\wedge \cdots \wedge df_p$. (Here we are identifying tensor fields with multi-variable functions as in section \ref{MVCRT}.)

Now consider the special case where $H=W_0$ is a finite group generated by reflections on $U=\Sigma$. If $W_0$ is irreducible of type $A$, $B$, or dihedral, then $\mathcal{P}^m=\R[\Sigma^m]^{W_0}$ by \cite{Weyl}, \cite{Hunziker97}.

It was noted by Wallach \cite{Wallach93} that $\R[\Sigma^m]^{W_0}$ is \emph{not} generated by polarizations for $W_0$ of type $D_n$ for $n>3$ and $m>1$. He proposed a definition of generalized polarizations, and showed that these do generate all multi-variable invariants for type $D$. Unfortunately Wallach's generalized polarizations fail to generate all multi-variable invariants for $W_0$ of type $F_4$ (see \cite{Hunziker97}).

For $W_0$ of general type, even though $\mathcal{P}^m \neq \R[\Sigma^m]^{W_0}$, one can still identify geometrically interesting subspaces of $\R[\Sigma^m]^{W_0}$ which are contained in $\mathcal{P}^m$. For example, Solomon's Theorem \cite{Solomon63} states that the subspace $\R[\Sigma, \Lambda^{m-1} \Sigma^*]^{W_0} \subset \R[\Sigma^m]^{W_0} $ of exterior $(m-1)$-forms is contained in $\mathcal{P}^m$. Another example is the main ingredient in the proof of Theorem \ref{metric}:
\begin{theorem}[Hessian Theorem --- \cite{Mendes11}]
\label{hessian}
Let $W_0\subset O(\Sigma)$ be a finite group generated by reflections. Then every $W_0$-invariant symmetric $2$-tensor field on $\Sigma$ is a sum of terms of the form $a$Hess$(b)$, for $a,b\in \R[\Sigma]^{W_0}$.
\end{theorem}

For the convenience of the reader, we provide a proof of Theorem \ref{hessian} above in the Appendix.

Now assume $K\subset O(V)$ is a polar representation of the compact group $K$ with section $\Sigma$, and generalized Weyl group $W$. Recall that the connected component of the identity $K_0$ is polar with the same section $\Sigma$, and denote by $W_0$ its generalized Weyl group. By \cite{Dadok85}, $W_0$ is a finite group generated by reflections. Since the operators $D_{ij}$ commute with the restriction map $|_{\Sigma^m}: \R[V^m]^{K_0}\to \R[\Sigma^m]^{W_0}$, and the single-variable invariants coincide by the Chevalley Restriction Theorem, the image of $|_{\Sigma^m}$ must contain $\mathcal{P}^m$. In particular, this gives an alternative proof of Theorem \ref{Tevelev} in the special case that $W_0$ is of classical type -- see \cite{Hunziker97}.

Similarly, Theorem \ref{hessian} implies surjectivity of the restriction map for symmetric $2$-tensors. In fact, we have the sharper statement:
\begin{lemma}
Let $K\subset O(V)$ be a polar representation of the compact group $K$, with section $\Sigma\subset V$ and generalized Weyl group $W$. Consider the restriction map for symmetric $2$-tensor fields $|_\Sigma : C^\infty(\mathrm{Sym}^2 V)^K\to C^\infty(\mathrm{Sym}^2 \Sigma)^W$.

This map is surjective. Moreover, given $\beta \in C^\infty(\mathrm{Sym}^2 \Sigma)^W$ there is $\tilde{\beta} \in C^\infty(\mathrm{Sym}^2 V)^K $ such the $\tilde{\beta}|_\Sigma=\beta$ and satisfying the following property:

For all $q\in V$, and $X,Y\in T_qV$ such that $X$ is vertical (that is, tangent to the $K$-orbit through $q$) and $Y$ is horizontal (that is, normal to the $K$-orbit through $q$), we have $\tilde{\beta}(X,Y)=0$.
\end{lemma}
\begin{proof}
Let $K_0$ be the connected component of the identity. It is polar with the same section $\Sigma$, and generalized Weyl group $W_0$. By \cite{Dadok85}, $W_0$ is generated by reflections.

Let $\beta\in C^\infty(\mathrm{Sym}^2 \Sigma)^W$. By Theorem \ref{hessian} together with \cite{Field77}, Lemma 3.1, $\beta$ is of the form $\beta=\sum_i a_i \mathrm{Hess}(b_i)$, where $a_i, b_i\in C^\infty(\Sigma)^{W_0}$. By the Chevalley Restriction Theorem, $a_i,b_i$ extend uniquely to $\tilde{a}_i, \tilde{b}_i\in C^\infty (V)^{K_0}$.

Define $\tilde{\beta}_0=\sum_i \tilde{a}_i\mathrm{Hess}(\tilde{b}_i)$ and 
$$ \tilde{\beta}=\frac{1}{|K/K_0|}\sum_{h\in K/K_0} h \tilde{\beta_0}$$
Then $\tilde{\beta}|_\Sigma=\beta$ by the same argument as in Corollary \ref{polarMVCRT}.

To show that $\tilde{\beta}$ satisfies the additional property in the statement of the Lemma, it is enough to do so for each Hess$(\tilde{\beta}_i)$. Changing the section $\Sigma$ if necessary, we may assume that $q,Y\in \Sigma$. Extend the given $X,Y\in T_qV$ to parallel vector fields (in the Euclidean metric), also denoted by $X,Y$. Let $f=d\tilde{\beta}_i(X)$.

We claim that $f|_\Sigma$ is identically zero. Indeed, since $X(q)$ is vertical, it is orthogonal to $\Sigma$, and so $X(p)$ is orthogonal to $\Sigma$ for every $p\in\Sigma$. Thus, for regular $p\in\Sigma$, $X(p)$ is vertical. Since $\tilde{\beta}_i$ is constant on orbits, $f(p)=0$ for every regular $p\in\Sigma$, and hence on all of $\Sigma$ by continuity.

Therefore Hess$(\tilde{\beta}_i)(X,Y)= df(Y)=0$, because $Y\in\Sigma$. 
\end{proof}

The following Lemma is needed in the proof of Theorem \ref{metric}.
\begin{lemma}
\label{positive}
Let $V$ be a polar $K$-representation with section $\Sigma \subset V$ and generalized Weyl group $W$. Let $\tilde{\sigma} \in C^\infty(\mathrm{Sym}^2 V)^K$, and $\sigma=\tilde{\sigma}|_\Sigma$. Then $\sigma(0)$ is positive definite if and only if $\tilde{\sigma}(0)$ is positive definite.
\end{lemma}
\begin{proof}
Denote by $K_0$ the connected subgroup of $K$ containing the identity. Recall that the action of $K_0$ is polar with the same section $\Sigma$. Denote by $W_0$ its generalized Weyl group. Consider a decomposition of $V$ into $K_0$-invariant subspaces
$$ V=\R^m\oplus V_1\oplus\cdots\oplus V_k$$
where $K_0$ acts trivially on $\R^m$, and each $V_i$ is irreducible and non-trivial.

By Theorem 4 in \cite{Dadok85}, each $V_i$ is a polar $K_0$-representation, with section $\Sigma_i=\Sigma\cap V_i$, and we have the decomposition into $W_0$-invariant subspaces
$$ \Sigma = \R^m\oplus \Sigma_1\oplus\cdots\oplus \Sigma_k$$
Moreover $W_0$ splits as a product $W_1\times\cdots\times W_k$ (see section 2.2 in \cite{Humphreys}), where $W_i$ is the generalized Weyl group associated to the section $\Sigma_i \subset V_i$, so that $\Sigma_i$ are pairwise inequivalent as $W_0$-representations. This implies that $V_i$ are pairwise inequivalent as $K_0$-representations.

Since the quotients $V_i/K_0$ and $\Sigma_i/W_0$ are isometric, irreducibility of $V_i$ as a $K_0$-representation implies irreducibility of $\Sigma_i$ as a $W_0$-representation. (Indeed, a general representation of a compact group $H$ on Euclidean space $\R^n$ is irreducible if and only if the quotient $S^{n-1}/H$ has diameter less than $\pi/2$)

By Schur's Lemma together with the assumption $\tilde{\sigma}|_\Sigma = \sigma$, 
$$ \sigma(0)=A\oplus\lambda_1 \mathrm{Id}_{\Sigma_1}\oplus\cdots\oplus \lambda_k\mathrm{Id}_{\Sigma_k}$$
$$ \tilde{\sigma}(0)=A\oplus\lambda_1 \mathrm{Id}_{V_1}\oplus\cdots\oplus \lambda_k\mathrm{Id}_{V_k}$$
where $A$ is a symmetric $m\times m$ matrix, and $\lambda_i \in\R$.

Therefore $\sigma(0) >0$ if and only if $\tilde{\sigma}(0) >0$.
\end{proof}

Now we are ready to prove Theorem \ref{metric}:
\begin{proof}[Proof of Theorem \ref{metric}]
As in the proof of Theorem \ref{mainthm}, we use partitions of unity and the Slice Theorem to reduce to the case where $M$ is a tube $\mathcal{U}=G\times_K V$, and $V$ is a polar representation. Let $\Sigma \subset V$ be a section, with generalized Weyl group $W$, so that $M/G=V/K=\Sigma/W$.

Note that it suffices to extend the given Riemannian metric $\sigma \in C^\infty(\mathrm{Sym}^2\Sigma)^W$ to a $G$-invariant Riemmanian metric on a possibly smaller tube $G\times_K V^\epsilon$ around the orbit $G/K$, for some $\epsilon >0$.

By Corollary \ref{symmetrizer}, $\sigma$ extends to $\beta_1\in C^\infty (\mathrm{Sym}^2 V)^K$. By Lemma \ref{positive}, $\beta_1(0)$ is positive-definite, and so by continuity, $\beta_1 >0$ on $V^\epsilon$ for some small $\epsilon>0$.

Choose any smooth, $K$-invariant and positive-definite $\beta_2 :V \to \mathrm{Sym}^2(T_KG/K)$. Then, by Lemma \ref{tube}, the pair $(\beta_1, \beta_2)$ defines $\tilde{\sigma}\in C^\infty(\mathrm{Sym}^2M)^G$, which is positive-definite on $G\times_K V^\epsilon$ and extends the given $\sigma$. By construction, $\Sigma$ is $\tilde{\sigma}$-orthogonal to $G$-orbits.
\end{proof}

\appendix
\section{Appendix --- Hessian Theorem for finite reflection groups}

In this section we provide a proof of Theorem \ref{hessian} for all finite reflection groups $W\subset O(\Sigma)$. Note that as far as the proof of Theorem \ref{metric} is concerned, the only case of Theorem \ref{hessian} needed is that of crystallographic reflection groups (see \cite{Humphreys} for a definition). Our proof includes the non-crystallographic case for the sake of completeness.

The structure of the proof is as follows. First we reduce to the case where $W$ is irreducible --- see Lemma \ref{product}. Then we point out that for $W$ irreducible of classical type, Theorem \ref{hessian} follows from more general polarization results due to Weyl \cite{Weyl} and Hunziker \cite{Hunziker97}. Finally we tackle the case of the exceptional groups with the help of a computer.

Recall some facts about finite reflection groups: First, the algebra of invariants, $\R[\Sigma]^W$, is a free polynomial algebra with as many generators as the dimension of $\Sigma$. This is known as Chevalley's Theorem --- see \cite {Bourbaki} Chapter V. Such a set of homogeneous generators is called a set of \emph{basic invariants}. Second, $\Sigma$ is reducible as a $W$-representation if and only if $\Sigma=\Sigma_1\times \Sigma_2$ and $W=W_1\times W_2$ for two reflection groups $W_k\subset O(\Sigma_k)$ --- see section 2.2 in \cite{Humphreys}. Because of the latter, the following proposition reduces the proof of the Hessian Theorem to the irreducible case.

\begin{lemma}
\label{product}
Let $W_k\subseteq O(\Sigma_k)$, $k=1,2$ be two finite reflection groups in the Euclidean vector spaces $\Sigma_k$, and let $W=W_1\times W_2\subset O(\Sigma)=O(\Sigma_1\times \Sigma_2)$. Then the conclusion of the Hessian Theorem holds for $W\subset O(\Sigma)$ if and only if it holds for both  $W_k\subseteq O(\Sigma_k)$, $k=1,2$.
\end{lemma}

\begin{proof}
Let $i_k:\Sigma_k\to \Sigma_1\times \Sigma_2$ and $p_k:\Sigma_1\times \Sigma_2\to \Sigma_k$ be the natural inclusions and projections. As a $W$-representation, $\mathrm{Sym}^2(\Sigma^*)$ decomposes as
$$ \mathrm{Sym}^2(\Sigma^*) =\mathrm{Sym}^2(\Sigma_1^*)\oplus \mathrm{Sym}^2(\Sigma_2^*)\oplus (\Sigma_1^*\otimes \Sigma_2^*)$$
Denote by $i_{11}$ and $i_{22}$ the natural inclusions of the first two summands. All these maps are $W$-equivariant.

Assume the conclusion of the Hessian Theorem holds for $W\subset O(\Sigma)$. Thus there are $Q_j\in\R[\Sigma]^W$ whose Hessians form a basis for $\R[\Sigma,\mathrm{Sym}^2(\Sigma^*)]^W$. Then the restrictions $Q_j|_{\Sigma_k}=i_k^*Q_j$ generate $\R[\Sigma_k,\mathrm{Sym}^2(\Sigma_k^*)]^{W_k}$ as an $\R[\Sigma_k]^{W_k}$-module.

Indeed, given $\sigma\in \R[\Sigma_k,\mathrm{Sym}^2(\Sigma_k^*)]^{W_k} $, define $$\tilde{\sigma}=i_{kk}\circ\sigma\circ p_k$$ Since $\tilde{\sigma}$ is $W$-equivariant, there are $a_j\in \R[\Sigma]^W$ such that $\tilde{\sigma}=\sum_j a_j\mathrm{Hess}(Q_j)$. Therefore
    $$ \sigma=i_k^*(\tilde{\sigma})=\sum_j (a_j|_{\Sigma_k})\mathrm{Hess}(Q_j|_{\Sigma_k})$$

For the converse, assume the conclusion of the Hessian Theorem holds for $W_k\subset O(\Sigma_k)$. Let $\rho_j\in\R[\Sigma_1]^{W_1}$, $j=1,\ldots n_1$ and $\psi_j\in\R[\Sigma_2]^{W_2}$, $j=1,\ldots n_2$ be basic invariants on $\Sigma_1$ and $\Sigma_2$ respectively, and $Q_j\in\R[\Sigma_1]^{W_1}$, for $j=1,\ldots (n_1^2+n_1)/2$, $R_j\in\R[\Sigma_2]^{W_2}$, for $j=1,\ldots (n_2^2+n_2)/2$ be homogeneous invariants whose Hessians form a basis for the corresponding spaces of equivariant symmetric $2$-tensors.

Claim: The Hessians of the following set of $W=W_1\times W_2$-invariant polynomials on $\Sigma=\Sigma_1\times \Sigma_2$ form a basis for the space of equivariant symmetric $2$-tensors on $\Sigma$:
$$ \{Q_j\} \cup \{R_j\}\cup \{\rho_i\psi_j , \quad i=1\ldots n_1,\  j=1\ldots n_2\}$$

Indeed, $\R[\Sigma,\mathrm{Sym}^2(\Sigma^*)]^W$ decomposes as
$$ \R[\Sigma,\mathrm{Sym}^2(\Sigma_1^*)]^W\oplus \R[\Sigma,\mathrm{Sym}^2(\Sigma_2^*)]^W\oplus\R[\Sigma,\Sigma_1^*\otimes \Sigma_2^*]^W$$
The first two pieces are freely generated over $\R[\Sigma]^W$ by Hess$Q_j$ and Hess$R_j$. The third piece can be rewritten as $\R[\Sigma,\Sigma_1^*\otimes \Sigma_2^*]^W=\R[\Sigma_1,\Sigma_1^*]^{W_1}\otimes\R[\Sigma_2,\Sigma_2^*]^{W_2}$. By Solomon's Theorem \cite{Solomon63}, $\R[\Sigma_k,\Sigma_k^*]^{W_k}$ are freely generated by $d\rho_j$ and $d\psi_j$, so that $\R[\Sigma,\Sigma_1^*\otimes \Sigma_2^*]^W$ is freely generated by $d\rho_j\otimes d\psi_j$. To finish the proof of the Claim one uses the product rule
$$ \mathrm{Hess}(\rho_i\psi_j)= d\rho_i\otimes d\psi_j + \rho_i\mathrm{Hess}(\psi_j)+\psi_j\mathrm{Hess}(\rho_i) $$
\end{proof}

Irreducible finite reflection groups are classified by type --- see \cite{Humphreys}. For $W$ irreducible of type $A$, $B$ and dihedral, the statement of Theorem \ref{hessian} follows from \cite{Weyl}, \cite{Hunziker97}, while for type $D$, it follows from \cite{Hunziker97}, Theorem 3.1.

Finally we prove the Hessian Theorem for the six exceptional finite reflection groups $W\subset O(\Sigma)$ usually called by the names of their Dynkin diagrams: $H_3$, $H_4$, $F_4$, $E_6$, $E_7$ and $E_8$. Note that the subscript denotes the rank $n=$dim$(\Sigma)$. 
In all cases our proof relies on calculations performed by a computer running GAP 3 (see \cite{GAP3}) using the package CHEVIE, which ultimately rely only on integer arithmetic. For the actual code that was used, see 

http://www.nd.edu/\~{}rmendes/sym2.txt

Recall a way of describing $W\subset O(\Sigma)$ from its Cartan matrix $C=(C_{ij})$. $\Sigma$ has a basis $r_1,\ldots r_n$ of simple roots with corresponding co-roots $r^{\vee}_1,\ldots r^{\vee}_n$. This means that $W$ is generated by the reflections in the hyperplanes $\ker(r^{\vee}_i)$ given by:
$$ R_i :v \mapsto v-r^{\vee}_i (v) r_i \qquad i=1,\ldots n$$ 
Expressing $v\in \Sigma$ in the basis of simple roots $v=a_1r_1+\ldots a_nr_n$, we get
$$R_i(v)=v-\left(\sum_j a_jr_i^{\vee}(r_j)\right)r_i$$
The coefficients $r_i^{\vee}(r_j)=C_{ij}$ form the Cartan matrix.

Here are the Cartan matrices for $H_3$, $H_4$ and $F_4$: (where $\zeta = \exp(2\pi i/5)$)

$$H_3:\  \left( \begin{array}{ccc} 
2 & \zeta^2+\zeta^3 & 0 \\
\zeta^2+\zeta^3 & 2 & -1 \\
0 & -1 & 2 
\end{array}\right)  ,\quad
H_4:\ \left( \begin{array}{cccc} 
2 & \zeta^2+\zeta^3 & 0 &0 \\
\zeta^2+\zeta^3 & 2 & -1 & 0 \\
0 & -1 & 2 & -1 \\
0 & 0 & -1 & 2
\end{array}\right) $$ $$
F_4:\ \left( \begin{array}{cccc} 
\phantom{-}2 & -1 & \phantom{-}0 & \phantom{-}0 \\
-1 & \phantom{-}2 & -1 & \phantom{-}0 \\
\phantom{-}0 & -2 & \phantom{-}2 & -1 \\
\phantom{-}0 & \phantom{-}0 & -1 & \phantom{-}2
\end{array}\right)$$
For the Cartan matrices in type E, refer to the tables at the end of \cite{Bourbaki}.

We start the proof of the Hessian Theorem by describing how the program computes the polynomial
$$\frac{P_t(\R[\Sigma,\mathrm{Sym}^2\Sigma^*]^W)}{P_t(\R[\Sigma]^W)}$$
where $P_t(U)=\sum _{l=0}^\infty (\dim U_l) t^l$ denotes the Poincar\'e series of a graded vector space $U=\oplus _{l=0}^\infty U_l$.

We need to recall a few facts. Let $I$ be the ideal in $\R[\Sigma]$ generated by the homogeneous invariants of positive degree. The quotient $\R[\Sigma]/I$ is known to be isomorphic, as a $W$-representation, to the regular representation (see Theorem B in \cite{Chevalley55}), but it is also a graded vector space. Fixing an irreducible representation/character $\xi$, the Poincar\'e polynomial FD$_\xi(t) $ of the subspace of $\R[\Sigma]/I$ with components isomorphic to $\xi$ is called the \emph{fake degree} of $\xi$ . Moreover $\R[\Sigma]$ is isomorphic to $(\R[\Sigma]/I)\otimes\R[\Sigma]^W$ . Thus the Poincar\'e series of the vector subspace in $\R[\Sigma]$ given by the direct sum of all irreducible subspaces isomorphic to $\xi$ equals FD$_\xi(t) P_t(\R[\Sigma]^W)$.

The way the program computes $P_t(\R[\Sigma,\mathrm{Sym}^2\Sigma^*]^W)$ is as follows:

It first computes the character $\chi$ of Sym$^2\Sigma^*$, and decomposes it into a sum of irreducible characters, using character tables that come with CHEVIE.
$$ \chi=\sum_{\xi\text{ irreducible}} c_\xi \xi$$

It then uses a command in CHEVIE that returns the fake degrees of the irreducible characters $\xi$, and computes
$$ \sum_\xi c_\xi \mathrm{FD}_\xi(t)$$
Using Schur's Lemma one sees that this equals
$$\frac{P_t(\R[\Sigma,\mathrm{Sym}^2\Sigma^*]^W)}{P_t(\R[\Sigma]^W)}$$

Here are the outputs:
$$ \begin{array}{l|l}
 & P_t(\R[\Sigma,\mathrm{Sym}^2\Sigma^*]^W) / P_t(\R[\Sigma]^W)=\\
 \hline\\
H_3 & t^{10}+t^8+t^6+t^4+t^2+1\\
H_4 & t^{38}+t^{30}+t^{28}+t^{22}+t^{20}+t^{18}+t^{12}+t^{10}+t^2+1\\
F_4 & t^{14}+t^{12}+2t^{10}+t^8+2t^6+t^4+t^2+1\\
E_6 &  t^{16} + t^{15} + t^{14} + t^{13} + 2t^{12} + t^{11} +2t^{10} + 2t^9 + \\
&+2t^8 + t^7 + 2t^6 + t^5 + t^4 + t^3 + t^2 + 1\\
E_7 &  t^{26} + t^{24} + 2t^{22} + 2t^{20} + 3t^{18} + 3t^{16} +
3t^{14}  + 3t^{12} +\\
&+ 3t^{10} + 2t^8 + 2t^6 + t^4 + t^2 + 1\\
E_8 & t^{46} + t^{42} + t^{40} + t^{38} + 2t^{36} + 2t^{34} + t^{32} + 3t^{30} + 2t^{28} + 2t^{26} + 3t^{24} + \\
& +2t^{22} +  2t^{20} + 3t^{18} + t^{16} + 2t^{14} + 2t^{12} + t^{10} + t^8 + t^6 + t^2 + 1 
\end{array}$$

Now we turn to the task of defining an explicit set of basic invariants $\rho_1,\ldots \rho_n\in \R[\Sigma]^W$. The degrees $d_i=\mathrm{deg}(\rho_i)$ are well known: (see tables at the end of \cite{Bourbaki})
$$ \begin{array}{l|l}
&\text{degrees } d_1,\ldots d_n\\
\hline
H_3 & 2,6,10\\
H_4 & 2,12, 20, 30\\
F_4 & 2,6,8,12 \\
E_6 & 2,5,6,8,9,12\\
E_7 & 2,6,8,10,12,14,18\\
E_8 & 2,8,12,14,18,20,24,30
\end{array}$$

We choose for each group a regular vector $v\in \Sigma$ and identify it with the row vector of its coefficients in the basis of the simple roots $r_i$. We also take one non-zero $\lambda\in \Sigma^*$ with minimal $W$-orbit size, namely the one which in the basis $\{r^{\vee}_i\}$ of simple co-roots is identified with the row vector
$$ \lambda=(0,\ldots 0,1)\cdot C^{-1} $$

Then the program computes the $W$-orbit $\mathcal{O}$ of $\lambda$. Here are our choices of $v$ and the number of elements in the orbit $\mathcal{O}$:
$$ \begin{array}{l|l|l}
 & v \text{ (in the basis } \{r_i\}) &  |\mathcal{O}|\\
\hline
H_3 & (1,2,3) & 12\\
H_4 & (1,2,3,5) & 20\\
F_4 & (2,-3,5,7)  & 24\\
E_6 & (2,-5,41,7,-9,110) & 27\\
E_7 & (2,-5,41,7,-9,110 ,-87) & 56\\
E_8 & (2,-5,41,7,-9,110 ,-87,11) & 240
\end{array}$$

Since $W$ permutes the linear polynomials in $\mathcal{O}$, for each natural number $m$ we get a $W$-invariant polynomial of degree $m$
$$\psi_m=\sum_{\lambda \in \mathcal{O}} \lambda^m$$

The invariants constructed this way are called the Chern classes associated to the orbit $\mathcal{O}$. See \cite{NeuselSmith} chapter 4.

\begin{lemma}
The polynomials $\rho_i=\psi_{d_i}$, $i=1,\ldots n$, form a set of basic invariants, and $v$ is indeed a regular vector.
\end{lemma}
\begin{proof}
Let $J$ be the Jacobian matrix
$$J= \left( \frac{\partial\rho_i}{\partial r^{\vee}_j} \right)_{i,j}=\left( \sum_{\lambda \in \mathcal{O}} d_i\lambda^{d_i-1}\frac{\partial\lambda}{\partial r^{\vee}_j}\right)_{i,j} $$  The program computes its determinant, evaluates it at the vector $v$, and checks that the value is non-zero. This proves both that $\rho_i$ are algebraically independent (see Proposition 3.10 in \cite{Humphreys}) and hence a set of basic invariants because they have the right degrees; and that $v$ is indeed a regular vector, that is, does not belong to any of the reflecting hyperplanes (see section 3.13 in \cite{Humphreys}). \end{proof}

We point out that L. Flatto and M. Weiner studied the set of all $\lambda\in \Sigma^*$ that make the $\rho_i=\psi_{d_i}$ constructed above a set of basic invariants. They produce a distinguished set of basic invariants $J_1, \ldots J_n$, determined up to non-zero constants, such that $\lambda\in \Sigma^*$ gives rise to a set of basic invariants if and only if $J_i(\lambda)\neq 0$ for all $i$ --- see \cite{FlattoWeiner69,Flatto70} for more details.

\begin{theorem}
\label{exceptional}
Let $W\subset O(\Sigma)$ be one of the six exceptional finite reflection groups, and $\rho_1, \ldots \rho_n$ the set of basic invariants described above. Let $ T\subset \{\rho_i\} \cup \{\rho_i\rho_j\}$ be a subset with $n(n+1)/2$ elements such that $T$ contains $\{\rho_i\}$ and
$$\sum_{Q\in T} t^{\deg(Q)-2} =  \frac{P_t(\R[\Sigma,\mathrm{Sym}^2\Sigma^*]^W)}{P_t(\R[\Sigma]^W)}$$

There is at least one such $T$, and for each one,
$ \{ \mathrm{Hess}(Q)\ |\ Q\in T\}$ is a basis for $\R[\Sigma,\mathrm{Sym}^2\Sigma^*]^W$ as a free module over $\R[\Sigma]^W$.
\end{theorem}

\begin{proof}

First the program finds a list of all subsets $T$ satisfying the condition in the statement of the Theorem. The number of elements in this list (choices for $T$) are:

$$\begin{array}{l|l|l|l|l|l|l}
& H_3 & H_4 &  F_4  & E_6 & E_7 & E_8 \\
\hline
\text{choices }& 2 & 2 & 2 & 12 & 48 & 96\\
\end{array}$$

For each $T$, the program constructs a square matrix $M$ of size $n(n+1)/2$. The rows are in correspondence with the set $\mathcal{H}= \{ \mathrm{Hess}(Q)\ |\ Q\in T\}$ , and the columns with the set $\mathcal{P}$ of upper triangular positions of an $n\times n$ matrix. The entry of $M$ associated with $\mathrm{Hess}(Q)\in \mathcal{H}$ and a position $(a,b)\in\mathcal{P}$ is the $(a,b)$-entry of Hess$(Q)(v)$, that is, $$\frac{\partial^2 Q} { \partial r^{\vee}_a \partial r^{\vee}_b} (v)$$

Then it proceeds to compute the determinant of $M$ and checks that it is non-zero. This implies that $\mathcal{H}$ is linearly independent at $v$, hence over $\R[\Sigma]$, and in particular over $\R[\Sigma]^W$.

Therefore span$_{\R[\Sigma]^W}\mathcal{H}$ is a submodule of $\R[\Sigma,\mathrm{Sym}^2\Sigma^*]^W$ with the same Poincar\'e series, and so they must coincide.
\end{proof}

\bibliographystyle{plain}
\bibliography{ref}
\end{document}